\documentclass[]{amsart}
\usepackage{amsmath, amssymb}
\usepackage[colorlinks = true,
            linkcolor = blue,
            urlcolor  = red,
            citecolor = red,
            anchorcolor = red]{hyperref}
\usepackage{mathptmx}
\newcommand\numberthis{\addtocounter{equation}{1}\tag{\theequation}}

\newlength{\bibitemsep}
\newlength{\bibparskip}
\let\oldthebibliography\thebibliography
\renewcommand\thebibliography[1]{%
	\oldthebibliography{#1}%
	\setlength{\parskip}{\bibitemsep}%
	\setlength{\itemsep}{\bibparskip}%
}
\newtheorem{thm}{Theorem}[section]
\newtheorem{cor}[thm]{Corollary}
\newtheorem{lem}[thm]{Lemma}

\newtheorem{prop}[thm]{Proposition}
\numberwithin{equation}{section} 
\theoremstyle{definition}

\newtheorem{rem}[thm]{Remark}

\newtheorem{eg}[thm]{Example}

\newcommand{\HS}{\mathcal{H}}
\newcommand{\LS}{\mathcal{L}}
\newcommand{\KS}{\mathcal{K}}
\newcommand{\B}{\mathcal{B}}
\newcommand{\V}{\mathcal{V}}
\newcommand{\M}{\mathcal{M}}

\newcommand{\x}{\mathbf{x}}
\newcommand{\y}{\mathbf{y}}

\title[Distance and best approximations in operator norm and trace class norm]{Distance and best approximations in operator norm and trace class norm}

\author{Saikat Roy}
\address{Saikat Roy, Department of Mathematics, Indian Institute of Science, Bengaluru 560012, Karnataka, India.}
\email{saikatroy.cu@gmail.com}

\keywords{Best approximations, Trace class operators, Compact operators; Finite-dimensional $C^*$-algebra}
\subjclass[2020]{Primary 46B28, 46B02 Secondary 47L05.}
\thanks{}

\begin{document}

\maketitle

\begin{abstract}
We study the best approximation and distance problems in the operator space $\B(\HS)$ and in the space of trace class operators $\LS^1(\B(\HS))$. Formulations of distances are obtained in both cases. The case of finite-dimensional $C^*$-algebras is also considered. The computational advantage of the results is illustrated through examples.
\end{abstract}

\tableofcontents

\section{Introduction}

\noindent Throughout the article $\HS$ denotes an infinite-dimensional separable Hilbert space over $\mathbb{C}$, equipped with an inner product $\langle \cdot, \cdot \rangle$. The symbol $S_\HS$ stands for the unit sphere of $\HS$. For any sequence $(\zeta_n)\subset \HS$, we write $\zeta_n\xrightarrow{w} \zeta_0$ to mean $(\zeta_n)$ weakly converges to $\zeta_0$. Let $\B(\HS)$ denote the $C^*$-algebra of bounded linear operators on $\HS$, $\KS(\HS)$, the ideal of compact operators, and $\LS^1(\B(\HS))$, the ideal of trace class operators of $\B(\HS)$. If $\V$ is a finite-dimensional subspace of $\B(\HS)$, we define
\[
N_\V := \{a\in \B(\HS):~Tr(ay)=0,~y\in \V\} = \{a\in \B(\HS):~Tr(ay_j)=0,~1\leq j \leq k\}
\]
for any basis $\{y_1,y_2,\dots, y_k\}$ of $\V$. For any two vectors $\xi$ and $\zeta$ in $\HS$, $\xi\otimes \zeta$ denotes the rank-one operator $\eta\mapsto \langle \eta, \zeta\rangle \zeta$ ($\eta\in \HS$). spectrum of an element $x\in \B(\HS)$ is denoted by $\sigma(x)$.

\subsection{Motivation} A fundamental result due to Holmes and Kripke \cite{HK} asserts that $\KS(\HS)$ is a proximinal subspace of $\B(\HS)$ and for any $x\in \B(\HS)\setminus \KS(\HS)$, the distance from $x$ to $\KS(\HS)$ is given by $dist(x,\KS(\HS))=\Delta(x)$, where $\Delta(x):=\sup\{\lim_n\sup~\|x\zeta_n\|:~(\zeta_n)\subset S_\HS,~\zeta_n\xrightarrow{w} \mathbf{0}\}$ (see also \cite{ABJS, GK}). It is natural to ask whether $dist(x,\V)=\Delta(x)$, if we consider a proper subspace $\V$ of $\KS(\HS)$ instead of the whole space $\KS(\HS)$. The answer to this question is negative, as shown in Example \ref{Holmes Example}. Thus, the quantity $\Delta(x)$ becomes a lower bound of $dist(x,\V)$. This leads to the question: what is the exact formulation of $dist(x,\V)$ for $x\in \B(\HS)\setminus \KS(\HS)$ and a proper subspace $\V$ of $\KS(\HS)$? In this article, we formulate $dist(x,\V)$ in two cases. In the first case, we consider $\V$ to be any finite-dimensional subspace of $\B(\HS)$, and this answers the above question in the general framework of $\B(\HS)$ which in particular applies to the case of $\KS(\HS)$. In the second case, we consider the distance problem in the trace class norm when $x\in \LS^1(\B(\HS))$ and $\V$ is any finite-dimensional subspace of $\LS^1(\B(\HS))$.

\subsection{Approach} A general formulation of distance from a point $x$ to a subspace $Y$ in a Banach space $X$ is given by the classical duality principle \cite{Singer}. In \cite{Sain}, an alternative approach is presented for \emph{real} reflexive Banach spaces, sometimes under the assumption of strict convexity, which is focused on norm attainment of functionals. Such an approach is computationally advantageous as pointed out in \cite[Section 4] {Sain}, \cite{SR}. Here we consider the duality relations between $\KS(\HS)$, $\LS^1(\B(\HS))$ and $\B(\HS)$, and utilize the $\sigma$-WOT ($\sigma$-weak operator topology) compactness of the closed unit ball of $\B(\HS)$ to prove our results. Our methods are based on Banach space theoretic and elementary operator algebraic techniques. 

\medskip

\noindent The space $\B(\HS)$ is identified to the dual of $\LS^1(\B(\HS))$ through the isometric isomorphism $\Omega: \B(\HS)\to \LS^1(\B(\HS))^*$, defined as
\[
\Omega(x)=\rho_x, \quad x\in \B(\HS),~ \text{where}~ \rho_x(a) = Tr(ax), \quad a\in \LS^1(\B(\HS)). \numberthis \label{dual trac.}
\]
Similarly, the space $\LS^1(\B(\HS))$ can be identified to the dual of $\KS(\HS)$ through the isometric isomorphism $\Gamma: \LS^1(\B(\HS))\to \KS(\HS)^*$, defined as
\[
\Gamma(x)=\tau_x, \quad x\in \LS^1(\B(\HS)),~ \text{where}~ \tau_x(a) = Tr(ax), \quad a\in \KS(\HS). \numberthis \label{dual comp.}
\]
Adapting such an approach, we formulate the distance between a point and a subspace in $\B(\HS)$ both under the usual operator norm and trace norm, in terms of trace.
\begin{thm}\label{Thm: Oper.}
Let $\HS$ be a Hilbert space and $\V$ be a finite-dimensional subspace of $\B(\HS)$. Let $x\in \B(\HS)\setminus \V$. Then
\begin{equation}\label{I}
\begin{aligned}
dist(x,\V) & =\sup\{|Tr(ax)|:~a\in N_\V\cap \LS^1(\B(\HS)),~\|a\|_1=1\}\\
& \geq \sup\{|Tr((\xi\otimes \zeta)x)|:~(\xi\otimes \zeta)\in N_\V, \|\xi\|=\|\zeta\|=1\}.
\end{aligned}
\end{equation}
\end{thm}

\noindent Applying Theorem \ref{Thm: Oper.} we get a characterization of best approximation to $x$ out of $\V$.

\begin{thm}\label{Best approx. B(H)}
Let $\HS$ be a Hilbert space and $\V$ be a finite-dimensional subspace of $\B(\HS)$. Let $x\in \B(\HS)\setminus \V$. Then $z\in \V$ is a best approximation to $x$ out of $\V$ if and only if $\Omega(x-z)$ is a Hahn-Banach extension of $\Omega(x)\big|_{N_\V}$.
\end{thm}

\noindent For trace class operators, we have the following formulation of distance.
\begin{thm}\label{Thm: dist. trace}
Let $\HS$ be a Hilbert space and $\V$ be a finite-dimensional subspace of $\LS^1(\B(\HS))$ with $dim~\V=n$. Let $x\in \LS^1(\B(\HS))\setminus \V$. Then there exist scalars $\lambda_1, \dots, \lambda_k\in [0,1]$ $(1\leq k \leq 2n+1)$ with $\sum_{i=1}^k \lambda_i = 1$ such that
\begin{equation}
\begin{aligned}\label{Dist. Trace}
dist_1(x,\V) = & \sup \{|Tr(ax)|:~a\in \KS(\HS)\cap N_\V,~\|a\|=1\}\\
= & \max \{|Tr(ax)|:~a\in \B(\HS)\cap N_\V,~\|a\|=1\}\\
= & \max \{\sum_{i=1}^k \lambda_kTr(u_kx):~\sum_{i=1}^k 
Im~\lambda_iTr(u_ix)=0,~\sum_{i=1}^k \lambda_iTr(u_iw)=0,~v\in \V\\
& \qquad \quad u_i^*u_i=Id_\HS,~\text{or}~u_iu_i^*=Id_\HS,~\text{for each}~i\in \{1, \dots,k\}\}.
\end{aligned}
\end{equation}
\end{thm}

\noindent The $C^*$-homomorphism $Id\otimes \mathbf{1}:\B(\HS)\to \B(\HS\overline{\otimes}\ell_2(\mathbb{N}))$, defined as $id\otimes \mathbf{1}(x)=x\otimes \mathbf{1}$ isometrically embeds $\B(\HS)$ inside $\B(\HS\overline{\otimes}\ell_2(\mathbb{N}))$. The pull-back of weak operator topology on $\B(\HS\overline{\otimes}\ell_2(\mathbb{N}))$ under the map $Id\otimes \mathbf{1}$ is a locally convex topology, known as the $\sigma$-WOT topology on $\B(\HS)$. In particular, the $\sigma$-WOT continuous functionals are described by the following lemma.

\begin{lem}[\cite{Tak}, Chapter II]\label{Prelim: Trace class}
Let $\varphi:\B(\HS)\to \mathbb{C}$ be a linear functional. Then the following conditions are equivalent:
\begin{itemize}
    \item[(i)] There exists $a\in \LS^1(\B(\HS))$ such that 
\[
\varphi(x)=\varphi_a(x)=Tr(ax), \quad x\in \B(\HS).
\]
    \item[(ii)] $\varphi$ is $\sigma$-WOT continuous.
\end{itemize}
\end{lem}

\noindent Thus, $\sigma$-WOT continuous functionals are nothing but the weak$^*$ continuous functional on $\B(\HS)$. Since $\B(\HS)$ is a von-Neumann algebra, its closed unit ball is $\sigma$-WOT compact. Therefore, for any $x\in \LS^1(\B(\HS))$, the following collection
\[
M_x := \{a\in \B(\HS):~\|a\|=1,~\pi(x)(a)=\|x\|_1\}
\]
is always non-empty, where $\pi:\LS^1(\B(\HS))\to \LS^1(\B(\HS))^{**}$ denotes the canonical embedding. This observation leads to the following characterization of best approximations in trace class norm.

\begin{thm}\label{Thm: Trace. Oper.}
Let $\HS$ be a Hilbert space and $\V$ be any finite-dimensional subspace of $\LS^1(\B(\HS))$. Let $x\in \LS^1(\B(\HS))\setminus \V$. Then $y_0\in \V$ is a best approximation to $x$ out of $\V$ if and only if $M_{x-y_0}\cap N_\V\neq \emptyset$. 
\end{thm}
\noindent The above result has an important consequence for smooth trace class operators. Finally, we consider the finite-dimensional $C^*$-algebra case to formulate the distance of a point from a subspace in the $C^*$-norm and in the trace norm.
\begin{thm}\label{C-star case}
Let $n_1,n_2, \dots n_p$ be natural numbers and $\M$ be the finite-dimensional $C^*$-algebra
\[
M_{n_1}(\mathbb{C})\oplus_\infty M_{n_2}(\mathbb{C}) \oplus_\infty \dots \oplus_\infty M_{n_p}(\mathbb{C}).
\]
Let $\V$ be a subspace of $\M$ with $dim~\V=k$. Let $\x=(x_1,x_2,\dots,x_p)\in \M\setminus \V$. Then
\[
dist(\x,\V)=\max\left\{|\sum_{i=1}^p Tr(x_ia_i)|:~\sum_{i=1}^pTr(a_iy^{(j)}_i)=0,~\sum_{i=1}^p \|a_i\|_1=1,~1\leq j\leq k\right\},
\]
and
\[
dist_1(\x,\V)=\max\left\{|\sum_{i=1}^p Tr(x_ia_i)|:~\sum_{i=1}^pTr(a_iy^{(j)}_i)=0,~\max\{\|a_i\|:~1\leq i\leq p\}=1,~1\leq j\leq k\right\},
\]
for any basis $\{\y_j\in \M:\y_j=(y^{(j)}_1,y^{(j)}_2,\dots,y^{(j)}_p),1\leq j \leq k\}$ of $\V$.
\end{thm}
\noindent The applications of the above results are illustrated through suitable examples.

\medskip

\noindent The following two results will be used in the sequel. For a normed linear space $X$, we denote the unit sphere of $X$ by $S_X$ and the continuous dual of $X$ by $X^*$.

\begin{thm}[\cite{Singer} Theorem 1.1, Page 170]\label{Singer}
Let $X$ be a complex normed linear space and $G$ be an $n$-dimensional subspace of $X$. Let $x\in X\setminus G$ and $g_0\in G$. Then the following conditions are equivalent:
\begin{itemize}
    \item[(i)] $g_0$ is a best approximation to $x$ out of $G$.
    \item[(ii)] There exist $k$-number of extremal points $f_1,\dots,f_k\in S_{X^*}$ $(1\leq k \leq 2n+1)$ and scalars $\lambda_1,\dots, \lambda_k\in [0,1]$ with $\sum_{i=1}^k \lambda_i = 1$ such that
    \[
    \sum_{i=1}^k \lambda_if_i(x-g_0) = \|x-g_0\| \quad \text{and} \quad  \sum_{i=1}^k \lambda_if_i(g) = 0, \quad g\in G.
    \]
\end{itemize}
\end{thm}
\noindent The above result is due to Singer and has been applied to study the best approximation problems in the space of compact operators on a Banach space in \cite{MP}.

\begin{lem}[\cite{Rudin}, Lemma 3.9]\label{Prelim: Kernel}
Let $X$ be a vector space and $f$, $f_1,f_2,\dots, f_k$ are linear functionals on $X$. Then the following conditions are equivalent:
\begin{itemize}
    \item[(a)] $\{x\in X:~f_j(x)=0,~1\leq j\leq k\}\subset \ker~f$.
    \item[(b)] $f\in span\{f_1,f_2,\dots,f_k\}$. 
\end{itemize}
\end{lem}

\section{Distance and Best approximations in $\B(\HS)$}

\noindent We begin with an example which shows $dist(x,\V)\neq \Delta(x)$ for $x\in \B(\HS)\setminus \KS(\HS)$ and a proper subspace $\V$ of $\KS(\HS)$.

\begin{eg}\label{Holmes Example}
Let $\HS=\ell_2 (\mathbb{N})$, the Hilbert space of all square summable sequences. Let $(\zeta_n)_{n\geq 0}$ be the standard orthonormal basis of $\HS$. Consider the subspace $\HS_0=span\{\xi_0, \xi_1, \xi_2\}$. Let $x=2~Id_{\HS_0}\oplus u$ for an isometry $u$ on $\HS_0^\perp$. Consider the finite rank operators $y=a\oplus \mathbf{0}$ and $z=b\oplus \mathbf{0}$ with respect to the decomposition $\HS_0\oplus \HS_0^\perp$, where
\[
a = \begin{bmatrix}
1 & \frac{1}{2} & 0\\
0 & \frac{1}{2} & 1\\
1 & 0           & 0
\end{bmatrix}
\quad \text{and} \quad
b = \begin{bmatrix}
0           & 1           & -1\\
0           & \frac{1}{2} & 0\\
\frac{1}{3} & \frac{1}{4} & 0
\end{bmatrix}.
\]
Evidently, $y,z\in \KS(\HS)$, however, $x\notin \KS(\HS)$, as it contains the unilateral shift.

\medskip

\noindent Consider the $2$-dimensional subspace $\V=span\{y,z\}$. Observe that $Tr(y(\zeta_2\otimes \zeta_2))=Tr(z(\zeta_2\otimes \zeta_2))=0$ and $Tr(x(\zeta_2\otimes \zeta_2))=2$. Granted Theorem \ref{Thm: Oper.}, we have $dist(x,\V)\geq 2$. However, $dist(x,\V) = 2$, since $\|x\|=2$. Thus, $\mathbf{0}$ is the best approximation to $x$ out of $\V$.

\medskip

\noindent We now compute $\Delta(x)$. Let $(\xi_n)$ be any sequence of unit vectors in $\HS$ that weakly converges to $\mathbf{0}$. Let $\xi_n=\nu_n+\eta_n$, where $\nu_n\in \HS_0$ and $\eta_n\in \HS_0^\perp$ ($n\in \mathbb{N}$). For any $\phi\in \HS_0^\perp$, we have
\[
\lim_n\langle \xi_n, \phi \rangle = \lim_n\langle \eta_n, \phi \rangle = 0.
\]
Similarly, for any $\phi\in \HS_0$, we have $\lim_n\langle \xi_n, \phi \rangle=\lim_n\langle \nu_n, \phi \rangle = 0$. Since $\HS_0$ is finite-dimensional, this implies $\lim_n\|\nu_n\|=0$. Consequently, we get
\begin{align*}
\langle x \xi_n, x\xi_n \rangle  = \langle x \nu_n, x\nu_n \rangle + \langle x \eta_n, x\eta_n \rangle = 4\|\nu_n\|^2 + \|\eta_n\|^2 \leq 4\|\nu_n\|^2 + 1.
\end{align*}
Thus, $\lim_n\sup\|x\xi_n\|\leq 1$. In particular, if $\nu_n=\mathbf{0}$ for all $n\in \mathbb{N}$, we have $\lim_n\sup\|x\nu_n\|=1$. Therefore, $1=\Delta(x)<2=dist(x,\V).$ \qed
\end{eg}

\noindent We now prove Theorem \ref{Thm: Oper.}.

\begin{proof}[Proof of Theorem \ref{Thm: Oper.}]
Since $\|\xi\otimes \zeta\|_1=1$ for all $\xi,\zeta\in S_\HS$, it is enough to prove the first equality of (\ref{I}). We consider an isometric transformation of the equality, by identifying $\B(\HS)$ as $\LS^1(\B(\HS))^*$ through the isometric isomorphism $\Omega: \B(\HS)\to \LS^1(\B(\HS))^*$, as defined in (\ref{dual trac.}). Under this identification, the first equality of (\ref{I}) transforms to the following form:
\[
\|\rho_x-\rho_v\| = \sup\{|\rho_x(a)|:~a\in N_\V\cap \LS^1(\B(\HS)),~\|a\|_1=1\}.
\]
We now define a linear functional $\rho$ on $N_\V$ by $\rho(b)=\rho_x(b)$. Then $\rho$ extends to a functional $\widetilde{\rho}$ on $\LS^1(\B(\HS))$ with the preservation of norm. Since $(\rho_x-\widetilde{\rho})$ vanishes on $N_\V$, by Lemma \ref{Prelim: Kernel}, we have $(\rho_x-\widetilde{\rho})=\rho_{z_0}$ for some $z_0\in \V$. However, then
\begin{align*}
\|\rho_x-\rho_{z_0}\| & =\|\widetilde{\rho}\|\\
& = \|\rho\|\\
& = \sup\{|\rho_x(a)|:~a\in N_\V\cap \LS^1(\B(\HS)),~\|a\|_1=1\}.
\end{align*}
Consequently, we have
\[
dist(x,\V) \leq \sup\{|Tr(ax)|:~a\in N_\V\cap \LS^1(\B(\HS)),~\|a\|_1=1\}.
\]

\medskip

\noindent Conversely, since $\Omega(\V)$ is finite-dimensional, the existence of a best approximant is always guaranteed. For any best approximation $\rho_{v_0}$ to $\rho_x$ out of $\Omega(\V)$, $\rho_{v_0}$ vanishes on $N_\V$. Consequently,
\begin{align*}
\|\rho_x-\rho_{v_0}\| & \geq \sup\{|\rho_x(a)-\rho_{v_0}(a)|:~a\in N_\V\cap \LS^1(\B(\HS)),~\|a\|_1=1\}\\
& = \sup\{|\rho_x(a)|:~a\in N_\V\cap \LS^1(\B(\HS)),~\|a\|_1=1\},
\end{align*}
which gives
\[
dist(x,\V) \geq \sup\{|Tr(ax)|:~a\in N_\V\cap \LS^1(\B(\HS)),~\|a\|_1=1\}.
\]
This completes the proof.
\end{proof}

\noindent Therefore, we have seen that whenever we consider the best approximation problem from a point $x\in \B(\HS)$ to a finite-dimensional subspace $\V$ of $\KS(\HS)$, the quantity $\Delta(x)$ is only a lower bound of $dist(x,\V)$ and $dist(x,\V)$ is given by Theorem \ref{Thm: Oper.}. However, in certain cases $dist(x,\V)=\Delta(x)$. We record this observation in the following corollary, which also formulates the norm of an element $x$, especially in case (b).

\begin{cor}\label{Cor: compact}
Let $\HS$ be an infinite-dimensional Hilbert space. Let $\V$ be any finite-dimensional subspace of $\KS(\HS)$. Let $x\in \B(\HS)\setminus \KS(\HS)$. In each of the following cases below: 
\begin{itemize}
    \item[(a)]  $x$ attains its norm everywhere,
    \item[(b)] $x$ attains its norm nowhere,
\end{itemize}
there holds
\begin{align*}
dist(x, \KS(\HS))  = & dist(x, \V)\\
= & \max\{|Tr(ax)|:~a\in \LS^1(\B(\HS))\cap N_\V,~\|a\|_1=1\}\\
=& \|x\|.
\end{align*}
\end{cor}

\begin{proof}
In the first case $x=\|x\| v$ for some isometry $v$ on $\HS$. For any $y\in \KS(\HS)$, we have $(\|x\| v - y)^*(\|x\| v - y)=\|x\|^2 - y'$ for some $y'\in \KS(\HS)$, since $\KS(\HS)$ is a closed two-sided ideal in $\B(\HS)$. Thus, $\|x\|^2\in \sigma((\|x\| v - y)^*(\|x\| v - y))$, and we have $\|\|x\|v-y\|\geq \|x\|$, i.e., $\mathbf{0}$ is a best approximation to $x$ out of $\KS(\HS)$. When $x$ fails to attain its norm, it follows from \cite[Lemma 2]{HK} that any norming sequence $(\xi_n)$ of $x$ ($(\xi_n)\subset S_\HS$, $\|x \xi_n\|\to \|x\|$), $\xi_n\xrightarrow{w} \mathbf{0}$. Therefore, $\Delta(x)=\|x\|$, and again $\mathbf{0}$ is a best approximation to $x$ out of $\KS(\HS)$. Thus, we have the desired conclusion by Theorem \ref{Thm: Oper.}.
\end{proof}

\noindent Using Theorem \ref{Thm: Oper.}, we now prove Theorem \ref{Best approx. B(H)}.

\begin{proof}[Proof of Theorem \ref{Best approx. B(H)}]
If $z\in \V$ is a best approximation to $x$ out of $\V$, then $\Omega(x-z)\big|_{N_\V}=\Omega(x)\big|_{N_\V}$, as $\Omega(z)\in \Omega(\V)$ vanishes on $N_\V$. Now, by Theorem \ref{Thm: Oper.}
\begin{align*}
\|\Omega(x)\big|_{N_\V}\| & = \sup\{|Tr(ax)|:~a\in N_\V\cap \LS^1(\B(\HS)),~\|a\|_1=1\}\\
& =  dist(x,\V)\\
& = \|x-z\|\\
& =\|\Omega(x-z)\|.
\end{align*}
Conversely, if $\Omega(z-x)$ is a Hahn-Banach extension of $\Omega(x)\big|_{N_\V}$, then by virtue of Theorem \ref{Thm: Oper.}
\begin{align*}
\|x-z\| & =\|\Omega(x-z)\|\\
& = \|\Omega(x)\big|_{N_\V}\|\\
& = \sup\{|Tr(ax)|:~a\in N_\V\cap \LS^1(\B(\HS)),~\|a\|_1=1\}\\
& = dist(x,\V),
\end{align*}
and the proof is complete.
\end{proof}

\section{Distance and Best approximations in $\LS^1(\B(\HS))$ }

\noindent This section is devoted to the study of best approximations and distance problems in trace class norm. Similar to Theorem \ref{Thm: Oper.}, we can prove the first equality of Theorem \ref{Thm: dist. trace} by identifying $\LS^1(\B(\HS))$ to the dual of $\KS(\HS)$ through the isometric isomorphism $\Gamma: \LS^1(\B(\HS))\to \KS(\HS)^*$ as defined in (\ref{dual comp.}). Also, analogous to Theorem \ref{Best approx. B(H)}, we have a characterization of best approximations in the trace class. We only state the results as the proofs are similar to the proofs of Theorems \ref{Thm: Oper.} and \ref{Best approx. B(H)}.

\begin{prop}\label{sim. trace.}
Let $\HS$ be a Hilbert space and $\V$ be any finite-dimensional subspace of $\LS^1(\B(\HS))$. Let $x\in \LS^1(\B(\HS))\setminus \V$. Then
\begin{align*}
dist_1(x,\V) = \sup \{|Tr(ax)|:~a\in \KS(\HS)\cap N_\V,~\|a\|=1\}.
\end{align*}
\end{prop}

\begin{prop}
Let $\HS$ be a Hilbert space and $\V$ be a finite-dimensional subspace of $\B(\HS)$. Let $x\in \B(\HS)\setminus \V$. Then $z\in \V$ is a best approximation to $x$ out of $\V$ if and only if $\Gamma(x-z)$ is a Hahn-Banach extension of $\Gamma(x)\big|_{N_\V}$.
\end{prop}

\noindent However, utilizing the $\sigma$-WOT compactness of the closed unit ball of $\B(\HS)$, we prove Theorem \ref{Thm: Trace. Oper.} which also characterizes best approximations in trace class and has a substantial role in the sequel. We also use the notion of Birkhoff-James orthogonality  \cite{B,Ja} in the proof. For any two points $x$ and $y$ in a normed linear space $X$, $x$ is said to be Birkhoff-James orthogonal to $y$ (symbolized as $x\perp_B y$) if $\|x+\lambda y\|\geq \|x\|$ for all scalars $\lambda$. Birkhoff-James orthogonality is characterized by the well-known result of James \cite{Ja}, namely, 

\begin{lem}
Let $X$ be a normed linear space and $x,y\in X$ with $x\neq \mathbf{0}$. Then $x\perp_B y$ if and only if there exists $f\in S_{X^*}$ such that $f(x)=\|x\|$ and $f(y)=0$. 
\end{lem}
\noindent The collection 
\begin{align}\label{supp. func.}
J(x)=\{f\in S_{X^*}:~f(x)=\|x\|\}    
\end{align} 
is known as the collection of support functionals at $x$ which is always non-empty by the Hahn-Banach Theorem.

\begin{proof}[Proof of Theorem \ref{Thm: Trace. Oper.}]
We identify $\LS^1(B(\HS))$ to the space of $\sigma$-WOT continuous linear functionals on $\B(\HS)$ through the canonical embedding $\pi:\LS^1(\B(\HS))\to \LS^1(\B(\HS))^{**}$, defined as
\[
\pi(b)(z)=\varphi_z(b)=Tr(bz), \quad z\in \B(\HS).
\]
The collection $M_{x-y_0}$ is a face of the closed unit ball of $\B(\HS)$. Moreover, $M_{x-y_0}$ is, convex, $\sigma$-WOT compact subset of the unit sphere of $\B(\HS)$.

\medskip

\noindent If $M_{x-y_0}\cap N_\V= \emptyset$, by the Geometric Hahn-Banach Theorem, there exists a $\sigma$-WOT continuous linear functional $f:\B(\HS)\to \mathbb{C}$ and a positive real number $\alpha$ such that 
\[
Re~f(s)< \alpha < Re~f(z), \quad s\in N_\V,~z\in M_{x-y_0}.
\]
Since $N_\V$ is a subspace, we must have $Re~f(s)=0$ for all $s\in N_\V$. If $f(s_0)\neq 0$, for some $s_0\in S$, then $f(s_0)$ is purely imaginary. However, then $f(\iota s_0)\in \mathbb{R}$ and is non-zero, which contradicts that $Re~f(s)=0$ for all $s\in N_\V$. Therefore, we conclude $f(s)=0$ for all $s\in N_\V$. However, then $N_\V\subset \ker~f$ and $f\in \pi(\V)$ by Lemma \ref{Prelim: Kernel}.

\medskip

\noindent Let $f=\pi(v)$ for some $v\in \V$. Let $a\in \B(\HS)$ be a norm one element. Observe that $a$ is a support functional at $(x-y_0)$ if and only if
\[
\varrho_a(x-y_0)=Tr(a(x-y_0))=\|x-y_0\|_1=\pi(x-y_0)(a),
\]
i.e., $a\in M_{x-y_0}$. However, 
\[
Re~\varrho_a(v) = Re~Tr(av) = Re~\pi(v)(a) = Re~f(a)> \alpha, \quad a\in M_{x-y_0},
\]
which implies that no support function at $(x-y_0)$ vanishes at $w$. Consequently, by the James characterization $(x-y_0)\not\perp_B w$. Therefore, there exists $\lambda\in \mathbb{C}$ such that
\[
\|(x-y_0)+\lambda w\|_1 = \|x-(y_0-\lambda w)\|_1 < \|x-y_0\|_1,
\]
and $y_0$ fails to be a best approximation to $x$, a contradiction. Therefore, $M_{x-y_0}\cap N_\V\neq \emptyset$.

\medskip

\noindent Conversely, suppose that there exists $a_0\in M_{x-y_0}\cap N_\V$. Since $\|a_0\|=1$ and $Tr(a_0v)=0$ for all $v\in \V$, we have
\[
\|x-v\|_1\geq |Tr(a_0(x-v))| = |Tr(a_0x)| = |Tr(a_0x-a_0y_0)| = \|x-y_0\|_1, \quad w\in \V.
\]
Consequently, $y_0$ is the best approximation to $x$ out of $\V$.\\
This completes the proof.
\end{proof}

\noindent We now prove Theorem \ref{Thm: dist. trace}. We require the description of extreme points of the closed unit ball of $\LS^1(\B(\HS))^*$.

\begin{prop}\label{extr. dual. trace.}
$Ext(B_{\LS^1(\B(\HS))^*})= \{\rho_u:~x\in \B(\HS),~u^*u=Id_\HS~\text{or}~uu^*=Id_\HS\}.$
\end{prop}

\begin{proof}
The proof follows from the identification of $\Omega: \B(\HS)\to \LS^1(\B(\HS))^*$ and the fact that an extreme point of the closed unit ball of $\B(\HS)$ is either an isometry or a co-isometry \cite{Kad}.
\end{proof}

\begin{proof}[Proof of Theorem \ref{Thm: dist. trace}]
The first equality of (\ref{Dist. Trace}) follows from Proposition \ref{sim. trace.}. Since $\V$ is finite-dimensional, there exists $v_0\in \V$ such that $dist_1(x,\V)=\|x-v_0\|_1$. Observe that
\begin{align*}
dist_1(x,\V) & = \|x-v_0\|_1\\
& = \|\pi(x-y_0)\|\\
& = \sup \{|\pi(x-y_0)(a)|:~a\in \B(\HS),~\|a\|=1\}\\
& = \max \{|\pi(x-y_0)(a)|:~a\in \B(\HS)\cap N_\V,~\|a\|=1\} 
\end{align*}
since $M_{x-y_0}\cap N_\V\neq \emptyset$ by Theorem \ref{Thm: Trace. Oper.}. Thus, the second equality is proved.

\medskip

\noindent We now prove the third equality. It follows from Theorem \ref{Singer} and Proposition \ref{extr. dual. trace.} that there exist $w_1,\dots,w_k\in \B(\HS)$ $(1\leq k \leq 2n+1)$ and scalars $\lambda_1,\dots, \lambda_k\in [0,1]$ with $\sum_{i=1}^k \lambda_i = 1$ such that $w_i^*w_i=Id_\HS$ or $w_iw_i^*=Id_\HS$, for each $i\in \{1,\dots,k\}$, $\sum_{i=1}^k \lambda_i\rho_{w_i}(x-v_0) = \|x-v_0\|$ and $\sum_{i=1}^k \lambda_i\rho_{w_i}(v) = 0$ for all $v\in \V$. Thus,  $Im~\sum_{i=1}^k \lambda_i\rho_{w_i}(x-v_0)=Im~\sum_{i=1}^k \lambda_i\rho_{w_i}(x)=0$. In particular, the operator tuple $(w_1,\dots,w_k)\in \B(\HS)^k$ satisfies the following properties:
\begin{itemize}
    \item[(a)] $Im~\sum_{i=1}^k\lambda_i\rho_{w_i}(x)=0$
    \item[(b)] $\sum_{i=1}^k \lambda_i\rho_{w_i}(v) = 0$ for all $v\in \V$
    \item[(c)] $w_i^*w_i=Id_\HS$ or $w_iw_i^*=Id_\HS$, for each $i\in \{1,\dots,k\}$.
\end{itemize}
Additionally, $\sum_{i=1}^k \lambda_i\rho_{w_i}(x-v_0) = \sum_{i=1}^k \lambda_i\rho_{w_i}(x) = \|x-v_0\| = dist_1(x,\V)$. Let
\[
\Lambda := \{(u_1,\dots,u_k)\in \B(\HS)^k:~(u_1,\dots,u_k)~\text{satisfies}~(a),(b),(c)\}.
\]
Then $(w_1,\dots,w_k)\in \Lambda$ and $\max\{\sum_{i=1}^k \lambda_i\rho_{w_i}(x):~(w_1,\dots,w_k)\in\Lambda\} \geq dist_1(x,\V).$ However, the inequality is an equality, since for any $(u_1,\dots,u_k)\in \Lambda$, we have $\|\rho_{u_i}\|=\|u_i\|=1$ $(1\leq i \leq k)$, and therefore, $\sum_{i=1}^k \lambda_i\rho_{u_i}(x) = \sum_{i=1}^k \lambda_i\rho_{u_i}(x-v_0) \leq \|x-v_0\|_1 = dist_1(x, \V).$
Consequently, using (\ref{dual trac.}) we get
\[
\max\{\sum_{i=1}^k\lambda_i\rho_{u_i}(x):~(u_1,\dots,u_k)\in \Lambda\} = \max\{\sum_{i=1}^k\lambda_iTr(u_i x):~(u_1,\dots,u_k)\in \Lambda\} = dist_1(x, \V),
\]
and the third equality of (\ref{Dist. Trace}) follows. This completes the proof.
\end{proof}

\noindent The computational advantage of Theorem \ref{Dist. Trace} and Theorem \ref{Thm: Trace. Oper.} is illustrated through the following example.

\begin{eg}
Consider the trace class operators $x,y,z$ on $\ell_2(\mathbb{N})$ defined as
\begin{align*}
& x(\xi_0,\xi_1,\xi_2,\xi_3,\dots) = (\xi_0,\xi_0+\xi_1,\frac{\xi_1}{2},\frac{\xi_2}{2^2},\frac{\xi_3}{2^2}, \dots)\\
& y(\xi_0,\xi_1,\xi_2,\xi_3,\dots) = (0,2\xi_1,\xi_2,0,0,0, \dots)\\
& z(\xi_0,\xi_1,\xi_2,\xi_3,\dots) = (2\xi_0,0,-\xi_2,0,0,0, \dots)
\end{align*}
Let $\V= span\{y,z\}$. We wish to find $dist_1(x,\V)$ and a best approximation to $x$ out of $\V$.

\medskip

\noindent By Theorem \ref{Dist. Trace}, we have
\[
dist_1(x,\V)=\max \{|Tr(ax)|:~a\in \B(\HS)\cap N_\V,~\|a\|=1\}.
\]
Let $(\zeta_n)_{n\geq 0}$ be the standard orthonormal basis for $\ell_2(\mathbb{N})$. Observe that $a\in N_\V$ if and only if $Tr(ay)=Tr(az)=0$, i.e., $a\in N_\V$ if and only if
\begin{align}\label{criterion NV}
\langle a \zeta_0,  \zeta_0 \rangle = - \langle a \zeta_1, \zeta_1 \rangle = \frac{1}{2} \langle a \zeta_2,  \zeta_2 \rangle.
\end{align}
With this condition
\begin{align}\label{criterion trace}
|Tr(ax)| = |\langle a\zeta_0,\zeta_0 \rangle + \langle a\zeta_1,\zeta_1 \rangle + \sum_{n\geq 0}\frac{1}{2^n}\langle a\zeta_{n+1},\zeta_n \rangle| \leq 2.
\end{align}
However, the backward shift operator $\mathfrak{b}$ satisfies $(\ref{criterion NV})$ and gives $Tr(\mathfrak{b}x)=2$. Therefore, $dist_1(x,\V)=2$.

\medskip

\noindent Alternatively, we can use Theorem \ref{Thm: Trace. Oper.} to find a best approximation to $x$ out of $\V$ and then compute $dist_1(x,\V)$. Observe that $[x-\frac{1}{2}(y+z)](\xi_0,\xi_1,\xi_2,\xi_3,\dots)=(0,\xi_0,\frac{\xi_1}{2},\frac{\xi_2}{2^2},\frac{\xi_3}{2^2}, \dots)$ and
\[
Tr(\mathfrak{b}(x-\frac{1}{2}(y+z)))=\|x-\frac{1}{2}(y+z)\|_1=2.
\]
Since $\mathfrak{b}\in N_\V$, we have $M_{x-\frac{1}{2}(y+z)}\cap N_\V\neq \emptyset$. Therefore, by Theorem \ref{Thm: Trace. Oper.}, $\frac{1}{2}(y+z)$ is a best approximation to $x$ out of $\V$, and $dist(x,\V)=2$.
\end{eg}

\noindent Theorem \ref{Thm: Trace. Oper.} has a nice application in the smooth case. Smoothness in a normed linear space is characterized in terms of support functionals (see (\ref{supp. func.})). A non-zero point $x$ in a normed linear space $X$ is smooth if and only if $J(x)$ is a singleton set \cite{Ja, Roy}.

\begin{prop}\label{Prop. smooth}
Let $\HS$ be a Hilbert space and $x\in \LS^1(\B(\HS))$ be non-zero. Then $x$ is smooth with respect to the trace norm if and only if 
\[
\{a\in \B(\HS):~\|a\|=1,~Tr(ax)=\|x\|_1\} = \{v^*\},
\]
where $v$ is the unique partial isometry associated with the polar decomposition of $x$.
\end{prop}

\begin{proof}
By the James characterization of smoothness, $a$ is smooth if and only if $J(x)$ is a singleton set. Now, by (\ref{dual trac.}) we get
\[
J(x)=\{\rho_a:~a\in \B(\HS),~\|a\|=1,~\rho_a(x)=Tr(xa)=\|x\|_1\}.
\]
Thus, $x$ is smooth if and only if $\{a\in \B(\HS):~\|a\|=1,~Tr(ax)=\|x\|_1\}$ is a singleton set. Since $v^*\in \{a\in \B(\HS):~\|a\|=1,~Tr(ax)=\|x\|_1\}$, we have the desired equality.
\end{proof}

\begin{cor}\label{smooth. trace}
Let $\HS$ be a Hilbert space and $\V$ be a finite-dimensional subspace of $\LS^1(\B(\HS))$. Let $x\in \LS^1(\B(\HS))\setminus \V$ be smooth. Then $\mathbf{0}$ is the best approximation to $x$ out of $\V$ if and only if $v^*\in N_\V$, where $v$ is the unique partial isometry involved in the polar decomposition of $x$.
\end{cor}

\begin{proof}
It follows from Theorem \ref{Thm: Trace. Oper.} that $\mathbf{0}$ is the best approximation to $x$ out of $\V$ if and only if $M_x\cap N_\V\neq \emptyset$. However, since $x$ is smooth, by Proposition \ref{Prop. smooth}, we have $M_x=\{v^*\}$, and the assertion is now proved.
\end{proof}

\begin{cor}
Let $x,y\in M_n(\mathbb{C})$ for some natural number $n$ and $x\neq y$. Then
\[
dist_1(x,\mathbb{C}y)=\max\{|Tr(ax)|:~\|a\|=1,~Tr(ay)=0\}
\]
In particular, if $x=v|x|$ is the polar decomposition of $x$ with $Tr(v^*y)=0$, then
\[
dist_1(x,\mathbb{C}y)=\|x\|_1.
\]
\end{cor}

\begin{proof}
Let $\V$ be the one-dimensional space $\{\lambda y:~\lambda\in \mathbb{C}\}$. The first equality follows from Theorem \ref{Thm: dist. trace}. To prove the second equality, we observe that the partial isometry $v^*$ satisfies the hypothesis of the sufficient part of Theorem \ref{Thm: Trace. Oper.} with $y_0=\mathbf{0}$. Thus, $\mathbf{0}$ is the best approximation to $x$ out of $\V$. Consequently, we have $dist_1(x,\mathbb{C}y)=\|x\|_1$, and the proof is complete.
\end{proof}

\begin{rem}
The notion of Birkhoff-James orthogonality is directly connected to the best approximation problems in normed spaces. Namely, given any subspace $\mathbb{Y}$ of $\mathbb{X}$ and $x\in \mathbb{X}\setminus \mathbb{Y}$, $y_0\in \mathbb{Y}$ is a best approximation to $x$ out of $\mathbb{Y}$ if and only if $(x-y_0)\perp_B \mathbb{Y}$. Consequently, Theorems \ref{Best approx. B(H)} and \ref{Thm: Trace. Oper.} characterize Birkhoff-James orthogonality of a point and a subspace in respective norms.
\end{rem}

\section{Distance formula in Finite-dimensional $C^*$-Algebra}

\noindent This section aims to formulate the distance between a point and a subspace in a finite-dimensional $C^*$-algebra and its dual. A finite-dimensional $C^*$-algebra $\M$ is $\infty$-sum of a finite number of matrix algebras (up to isomorphism):
\[
\M\cong M_{n_1}(\mathbb{C})\oplus_\infty M_{n_2}(\mathbb{C}) \oplus_\infty \dots \oplus_\infty M_{n_p}(\mathbb{C})
\]
for natural numbers $n_1,n_2, \dots, n_p$. Observe that 
\[
\M^* \cong M_{n_1}(\mathbb{C})^*\oplus_1 M_{n_2}(\mathbb{C})^* \oplus_1 \dots \oplus_1 M_{n_p}(\mathbb{C})^*,
\]
where $M_{n_i}(\mathbb{C})^*$ is the space of all $n_i\times n_i$ matrices equipped with the trace norm. By reflexivity of $\M$, any member $\mathbf{z}=(z_1,z_2,\dots, z_p)$ of $\M$ corresponds to a linear functional $\varphi_\mathbf{z}$ acting on $\M^*$ such that
\[
\psi_\mathbf{z}(u_1,u_2,\dots, u_p)=\sum_{i=1}^p Tr(z_iu_i), \quad (u_1,u_2,\dots, u_p)\in \M^*.
\]
Moreover, such a correspondence is an isometry.

\noindent We now prove Theorem \ref{C-star case}.

\begin{proof}[Proof of Theorem \ref{C-star case}]
We prove the first equality, and the second equality follows from duality. We consider an isometric transformation of the stated equality, given by
\[
\min\{\|\psi_\x-\psi_\mathbf{v}\|:~\mathbf{v}\in \V\}=\max\{\psi_\x(\mathbf{a}):~\psi_{\y_j}(\mathbf{a})=0,~\|\mathbf{a}\|_1=1,~1\leq j\leq k\},
\]
where $\mathbf{a}=(a_1,a_2,\dots,a_p)$ and $\|\mathbf{a}\|_1=\sum_{i=1}^p \|a_i\|_1=1$. The above equality follows the same arguments about Hahn-Banach extension as in Theorem \ref{Thm: Oper.}. Indeed, we consider the Hahn-Banach extension $\widetilde{\rho}$ of $\rho$, where
\[
\rho:=\psi_\mathbf{x}\bigg|_{\bigcap_{\mathbf{v}\in \V} \ker \psi_\mathbf{v}}.
\]
Then $\psi_{\mathbf{z}_0}=\psi_\x-\widetilde{\rho}$ is a member of $\{\psi_\mathbf{v}:~\mathbf{v}\in \V\}$, and we have
\begin{align*}
\min\{\|\psi_\x-\psi_\mathbf{v}\|:~\mathbf{v}\in \V\} \leq  \|\psi_{\mathbf{z}_0}\| = \max\{\psi_\x(\mathbf{a}):~\psi_{\y_j}(\mathbf{a})=0,~\|\mathbf{a}\|_1=1,~1\leq j\leq k\}.
\end{align*}

\medskip

\noindent For the reverse equality, we consider any best approximation $\psi_{\mathbf{x}_0}$ to $\psi_\mathbf{x}$ out of $\{\psi_\mathbf{v}:~\mathbf{v}\in \V\}$. Then $\psi_{\mathbf{x}_0}$ vanishes on $\bigcap_{\mathbf{v}\in \V} \ker \psi_\mathbf{v}$, and this proves the reverse inequality and finishes the proof of the theorem.
\end{proof}

\begin{cor}
Let $\M=M_{2}(\mathbb{C})\oplus M_{2}(\mathbb{C}) \oplus \dots \oplus M_{2}(\mathbb{C})$ ($p$-times), and $\x=(x_1,x_2,\dots,x_p)$, where
\[
x_i=\begin{bmatrix}
x^{(i)}_{11} & x^{(i)}_{12}\\
x^{(i)}_{21} & x^{(i)}_{22}
\end{bmatrix}, \quad 1\leq i \leq p.
\]
and $\V$ be the subspace $D_{2}(\mathbb{C})\oplus D_{2}(\mathbb{C}) \oplus \dots \oplus D_{2}(\mathbb{C})$ ($p$-times), where $D_2$ denotes the subspace of diagonal matrices in $M_2(\mathbb{C})$. Then
\[
dist(x,\V) = \max\{|x_{12}^{(i)}|,|x_{21}^{(i)}|:~1\leq i \leq p\} \quad \text{and} \quad dist_1(x,\V)=\sum_{i=1}^p |x^{(i)}_{21}|+|x^{(i)}_{12}|.
\]
\end{cor}

\begin{proof}
 Let $\{\xi_1,\xi_2\}$ be the standard orthonormal basis of $\mathbb{C}^2$. It is easy to see that the collection of $p$-tuples
\[
\{(0,\dots,\xi_j\otimes \xi_j,\dots,0)~\}:~j\in\{1,2\}\}
\]
is a basis of $\V$ and 
\[
N_\V = \{(a_1,a_2,\dots,a_p):~diag{(a_i)}=0,~1\leq i\leq p\}.
\]
A straightforward computation reveals that
\begin{align*}
\mathcal{A}_1=& \left\{\mathbf{a} =(a_1,a_2,\dots,a_p)\in N_\V:~\sum_{i=1}^p\|a_i\|_1=1\right\}\\
= & \left\{\mathbf{a} =(a_1,a_2,\dots,a_p)\in N_\V:~\sum_{i=1}^p|a_{21}^{(i)}|+|a_{12}^{(i)}|=1\right\}\\
= & \left\{\mathbf{a} =(a_1,a_2,\dots,a_p)\in \M:~diag{(a_i)}=0,~\sum_{i=1}^p|a_{21}^{(i)}|+|a_{12}^{(i)}|=1\right\}
\end{align*}
Now,
\begin{align*}
dist(x,\V) = & \max\left\{|\sum_{i=1}^p Tr(x_ia_i)|:~\mathbf{a} =(a_1,a_2,\dots,a_p)\in \mathcal{A}_1\right\}\\
= & \max\left\{|\sum_{i=1}^p a_{12}^{(i)}x_{21}^{(i)}+a_{21}^{(i)}x_{12}^{(i)}|:~\sum_{i=1}^p|a_{21}^{(i)}|+|a_{12}^{(i)}|=1\right\}.
\end{align*}
Since any member $(\xi_1,\xi_2\dots,\xi_m)\in \ell_\infty^m$ acts on $\ell_1^m$ by pointwise multiplication, the quantity $dist(x,\V)$ is nothing but the norm of the functional $(x^{(1)}_{21},x^{(1)}_{12},\dots,x^{(p)}_{21},x^{(p)}_{12})\in \ell_\infty^{2p}$, which is given by
\[
dist(x,\V) = \max\{|x_{12}^{(i)}|,|x_{21}^{(i)}|:~1\leq i \leq p\}.
\]

\medskip

\noindent Similarly, 
\begin{align*}
\mathcal{A}_\infty=& \left\{\mathbf{a} =(a_1,a_2,\dots,a_p)\in N_\V:~\max\{\|a_j\|:~1\leq i\leq p\}=1\right\}\\
= & \left\{\mathbf{a} =(a_1,a_2,\dots,a_p)\in N_\V:~\max\{\max\{|a_{21}^{(i)}|,|a_{12}^{(i)}|\}:~1\leq i\leq p\}=1\}\right\}\\
= & \left\{\mathbf{a} =(a_1,a_2,\dots,a_p)\in \M:~diag{(a_i)}=0,~\max\{|a_{21}^{(i)}|,|a_{12}^{(i)}|:~1\leq i\leq p\}=1\right\}.
\end{align*}
Therefore,
\begin{align*}
dist_1(x,\V) = & \max\left\{|\sum_{i=1}^p Tr(x_ia_i)|:~\mathbf{a} =(a_1,a_2,\dots,a_p)\in \mathcal{A}_\infty\right\}\\
= & \max\left\{|\sum_{i=1}^p a_{12}^{(i)}x_{21}^{(i)}+a_{21}^{(i)}x_{12}^{(i)}|:~\max\{|a_{21}^{(i)}|,|a_{12}^{(i)}|:~1\leq i\leq p\}=1\right\},
\end{align*}
which is nothing but the norm of the functional $(x^{(1)}_{21},x^{(1)}_{12},\dots,x^{(p)}_{21},x^{(p)}_{12})\in \ell_1^{2p}$, as given by
\[
dist_1(x,\V)=\sum_{i=1}^p |x^{(i)}_{21}|+|x^{(i)}_{12}|.
\]
This completes the proof.
\end{proof}

\noindent {\bf Acknowledgements}\\

\noindent The research of Dr. Saikat Roy is supported by NBHM(DAE), Government of India, under the Mentorship of Prof. Apoorva Khare. The author thanks Prof. Debmalya Sain (IIIT Raichur) for helpful suggestions and encouragement while preparing the manuscript.

\end{document}